\documentclass[12pt, A4]{article} 
\usepackage{amsmath}
\usepackage{amsfonts}
\usepackage{amsthm}
\usepackage{amscd}
\usepackage{amssymb}
\usepackage{ascmac} 
\usepackage{graphicx}
\usepackage{bounddvi} 
\usepackage[all]{xy}

\newtheorem{theorem}{Theorem}[section]
\newtheorem*{theorem*}{Theorem}
\newtheorem{lemma}[theorem]{Lemma}
\newtheorem{proposition}[theorem]{Proposition} 
\newtheorem{corollary}[theorem]{Corollary} 
\theoremstyle{definition}
\newtheorem*{definition*}{Definition}

\newtheorem*{ques*}{Question}
\newtheorem{assertion}{Assertion}
\theoremstyle{remark}
\newtheorem{remark}{Remark}
\numberwithin{equation}{section}

\newcommand{\bZ}{\mathbb{Z}}
\newcommand{\bC}{\mathbb{C}}
\newcommand{\bR}{\mathbb{R}}
\newcommand{\bQ}{\mathbb{Q}}
\newcommand{\bF}{\mathbb{F}}
\newcommand{\fH}{\mathfrak{H}}
\newcommand{\bH}{\mathbb{H}}
\newcommand{\II}{\mbox{I\hspace{-1pt}I}}

\newcommand{\Ker}{{\rm Ker}\,}
\newcommand{\Iso}{{\rm Iso}}
\newcommand{\Ind}{{\rm Ind}}

\newcommand{\gr}[1]{\langle #1 \rangle}
\newcommand{\wt}{\widetilde}
\newcommand{\Aut}{{\rm Aut}}
\newcommand{\GL}{{\rm GL}}
\newcommand{\rank}{{\rm rank}}
\newcommand{\orth}{{\rm O}}
\newcommand{\Ga}{\Gamma}
\newcommand{\ga}{\gamma}
\newcommand{\al}{\alpha}

\newcommand{\sig}{\sigma}
\newcommand{\ol}{\overline}

\title{Determination of compact Lie groups with the Borsuk-Ulam property} 

\author{Ikumitsu Nagasaki}
\date{July 28, 2021}

\begin{document}
\maketitle

\begin{abstract}
 A compact Lie group $G$ is said to have the Borsuk-Ulam property if the Borsuk-Ulam theorem holds for $G$-maps between representation spheres. It is well-known that an elementary abelian $p$-group $C_p{}^n$ ($p$ any prime) and  an $n$-torus $T^n$, $n\ge 0$, have the Borsuk-Ulam property.
In this paper, we shall discuss the classical question of which compact Lie groups have the Borsuk-Ulam property and in particular we shall show that every extension group of an $n$-torus by a cyclic group of prime order does not have the Borsuk-Ulam property. This leads us that the only compact Lie groups with the Borsuk-Ulam property are $C_p^n$ and $T^n$, which is a final answer to the question. 
\end{abstract}


\section{Introduction} 
\label{intro}
The classical Borsuk-Ulam theorem \cite{Bor} has been generalised from various aspects. From the viewpoint of equivariant topology or transformation group theory, the following theorem is well-known as a generalisation of the Borsuk-Ulam theorem, see for example \cite{Bi}, \cite{FH}, \cite{M2}, etc.

\begin{theorem*}[Borsuk-Ulam type theorem]
Let $\Gamma$ be an elementary abelian $p$-group $C_p{}^n$ ($p$ any prime) or an $n$-torus $T^n$. For any fixed-point-free orthogonal $\Gamma$-representations $V$ and $W$, i.e., 
$V^\Ga =W^\Ga =0$, if there exists a $\Gamma$-map $f: S(V) \to S(W)$ between representation spheres, then the inequality $\dim V \le \dim W$ holds.  
\end{theorem*}

We say that a compact Lie group $G$ has the Borsuk-Ulam property if such a Borsuk-Ulam type result holds, and such $G$ is called a \emph{BU-group of type I}  according to \cite{N1}, that is,  
we call $G$ a BU-group of type I if the following property is satisfied: 
for any fixed-point-free orthogonal $G$-representations $V$ and $W$, if there exists a $G$-map $f: S(V) \to S(W)$ between representation spheres, then the inequality $\dim V \le \dim W$ holds. 
We also call $G$ a \emph{BU-group of type II} if the following property is satisfied:  
for any fixed-point-free orthogonal $G$-representations $V$ and $W$ with
the same dimension, if there exists a $G$-map $f : S(V) \to S(W)$, then the degree of $f$ is $\deg f\ne 0$.
It is also known that tori and elementary abelian $p$-groups are BU-groups of type \II, see \cite[Theorems 1 and 2]{M0}.
It is natural and interesting to ask the following question:
\begin{ques*}
Which compact Lie groups have the Borsuk-Ulam property? 
\end{ques*}
In previous studies of $G$-maps between spheres, several counterexamples of the Borsuk-Ulam property were sporadically found out; for example, $C_{pq}$, where $p$, $q$ are relatively prime positive integers, is not a BU-group of type $\II$ by \cite[p.60]{Br} and not of type I by a result of \cite{Wa}. As is observed in \cite{M1}, $S^3= S(\bH)$, $\bH$ the skew field of quaternions,  is not a BU-group of type I; in fact, the Hopf map $\pi: S^3 \to S^3/S^1\cong S^2$ is considered as an $S^3$-map.

In the early 1990s, this question and related topics were systematically studied in \cite{B} and \cite{M1}.  Consequently, it has been shown that almost compact Lie groups are neither BU-groups of type I nor of type \II.  However, some unsolved cases were left until recently. 
In \cite{N1}, we studied the remaining cases of finite groups and provided an answer to the above question in finite group case. In fact, we showed that the only finite BU-groups of type I [resp. \II] are the elementary abelian $p$-groups (including the trivial group). 

In this paper, we shall discuss the question in general compact Lie groups and provide a final answer to the question. 
As will be seen in the following sections, the question is reduced to the case of an extension of an $n$-torus by a cyclic group $C_p$ of prime order $p$, and then the following result will be shown.
\begin{proposition}\label{p1-1}
Let $G$ be any extension of an $n$-torus $T^{n}$ by a cyclic group $C_p$ of order $p$: 
\[1\to T^n\to G\to C_p\to 1,\quad  n\ge 1.\] 
Then $G$ is neither a BU-group of type I nor of type \II.   
\end{proposition} 
As a consequence, we obtain a final answer to the question as follows:
\begin{theorem}\label{t1-2}
The following statements are equivalent.
\begin{enumerate}
\item 
A compact Lie group $G$ is a BU-group of type I.
\item
A compact Lie group $G$ is a BU-group of type \II.
\item
A compact Lie group $G$ is isomorphic to an elementary abelian $p$-group $C_p{}^n$ or an $n$-torus $T^n$, $n\ge 0$.
\end{enumerate}
\end{theorem}
This theorem is deeply related to some results of \cite{M1}.
In \cite{M1}, a compact Lie group $G$ is said to have the Borsuk-Ulam property in the strong sense $A$ (or the property $\II A$) if the following property is satisfied:  for any $G$-map $f:S(V) \to S(W)$ with $\dim V = \dim W$ and $\dim V^G = \dim W^G$, if the degree $\deg f^G$ is prime to $|G/G_0|$, then $\deg f \ne 0$. 
Theorem \ref{t1-2} improves a result of \cite{M1} as follows:

\begin{corollary}
A compact Lie group $G$ has the Borsuk-Ulam property in the strong sense $A$ if
and only if $G$ is isomorphic to an elementary abelian $p$-group $C_p{}^n$ or an $n$-torus $T^n$, $n\ge 0$.
\end{corollary}
Indeed, if $G$ has the Borsuk-Ulam property in the strong sense $A$, then $G$ is a BU-group of type \II, since $\deg f^G =1$ when $S(V)^G=S(W)^G = \emptyset$. Therefore $G$ is isomorphic to an elementary abelian $p$-group $C_p{}^n$ or an $n$-torus $T^n$, $n\ge 0$. On the other hand, the converse is already shown by \cite[Theorems 1 and 2]{M0}.
\begin{remark}
Theorem 2 (b) of \cite{M1} is unfortunately incorrect. In particular, Lemma 1.2 of \cite{M1} does not hold for compact Lie groups with the property of type $\II A$. Theorem \ref{t1-2} also gives a negative answer to the conjecture of \cite{Bl} that $G$ has the Borsuk-Ulam property if and only if $G\cong T^n\times C_p{}^m$.
\end{remark}

\section{The property of BU-groups and reduction of cases}
Since the main focus of this paper is to construct counterexamples of the Borsuk-Ulam property, we give the following definition. 
\begin{definition*}
A compact Lie group $G$ is called an \emph{anti-BU-group} of type I  [resp. \II] if $G$ is not a BU-group of type I [resp. \II]. 
\end{definition*}

In \cite{N1}, we showed the fundamental property of BU-groups of each type. We restate it as follows:

\begin{proposition}[(\cite{N1})]\label{p2-1}
Let $G$ be a compact Lie group.
\begin{enumerate}
\item 
If there exists a quotient group $Q$ of $G$ being an anti-BU-group of type I [resp. $\II$], then  $G$ is an anti-BU-group of type I [resp. $\II$].
\item
If there exists a closed subgroup $H$ of finite index being an anti-BU-group of type I [resp. $\II$], then $G$ is an anti-BU-group of type I [resp. $\II$].
\end{enumerate}
\end{proposition}

Let $G_0$ denote the identity component of a compact Lie group $G$ and set $F=G/G_0$ a finite group.
By \cite[Theorem 2 and Proposition 2.2]{B}, connected compact Lie groups other than tori are anti-BU-groups of both types I and \II. On the other hand, by \cite{N1}, finite groups other than $C_p{}^m$ are anti-BU-groups of both types I and \II. Thus if $G_0$ is not a torus or if $F$ is not an elementary abelian $p$-group, then $G$ is an anti-BU-group of both types I and $\II$ by Proposition \ref{p2-1}. 
Therefore it suffices to consider the following type of extension
\[1\to T^n \to G \to C_p{}^m\to 1 \ (n\ge 1, m\ge1).\]
Furthermore, there exists a closed subgroup of index $p^{m-1}$ of such $G$; indeed, one can take $\pi^{-1}(C_p)$ for some $C_p\le C_p{}^m$, where $\pi: G\to C_p{}^m$ is the projection. Thus, by Proposition \ref{p2-1}, $G$ is reduced to the case of an extension
\begin{equation}\label{e2-1}
1\to T^n \to G \overset{\pi}{\to} C_p\to 1.
\end{equation}
In order to prove Proposition \ref{p1-1}, a further reduction is needed.
Here an $n$-torus $T^n$ is described by \[T^n = \{t=(t_1,\dots, t_n)\in \bC^n\,|\, |t_i|=1 \ (1\le i \le n)\}.\]
Let $a\in C_p$ denote a (fixed) generator of $C_p$. We here introduce
a semi-direct product:
 \[\Ga_p=T^{p-1} \rtimes_\rho C_p, \] 
where 
$\rho: C_p=\gr{a}\to \Aut(T^{p-1})$ is the homomorphism defined by 
\begin{equation}\label{e2-2}
\rho(a)(t)=(t_{p-1}^{-1}, t_1t_{p-1}^{-1}\dots, t_{p-2}t_{p-1}^{-1})
\end{equation}
for $t=(t_1, \dots t_{p-1})\in T^{p-1}.$
In fact, one can easily check that the $p$-fold composition of $\rho(a)$ with itself is $id$ and then $\rho$ defines the conjugate action of $C_p$ on $T^n$:
$$ata^{-1}=\rho(a)(t).$$
In particular, $\Ga_2$ is isomorphic to $\text{O}(2)$ the orthogonal group in dimension $2$.
The group $\Ga_p$ is a split extension of $T^{p-1}$ by $C_p$:
\[1\to T^{p-1}\to \Ga_p\to C_p \to 1,\] 
and plays an important role in the proof of Proposition \ref{p1-1}.
The remainder of this section will be devoted to showing the following:
\begin{proposition}\label{p2-2}
For any extension {\rm (\ref{e2-1})}, there exists a closed normal subgroup $N$ such that
$G/N$ is isomorphic to $S^1\times C_p$ or $\Gamma_p$.  
\end{proposition}

We begin by observing the following fact.
\begin{lemma}\label{l2-3}
If $C_p$ acts trivially on $T^n$ for extension {\rm (\ref{e2-1})}, then $G$ is abelian.
\end{lemma}
\begin{proof}
Let $\pi: G\to C_p$ be the projection and $\al\in G$ an element such that $\pi(\al)=a$.
Let \[G=\coprod_{i=0}^{p-1} \al^iT^n\] be a coset decomposition.
Since the $C_p$-action on $T^n$ is trivial, it follows that $\al^it\al^{-i} = t$ for any $t\in T^n$ and $0\le i\le p-1$. Thus, 
for any $g=\al^it$, $h= \al^js \in G$, where $t$, $s\in T^n$, 
\[gh = \al^it \al^js = \al^{i+j}ts = \al^js\al^it= hg.\]
Thus $G$ is abelian.
\end{proof}
By the following lemma, any extension (\ref{e2-1}) is reduced to a split extension.
\begin{lemma}\label{l2-4}
For any extension {\rm (\ref{e2-1})}, 
there exists a closed normal subgroup $N$ of $G$ such that $G/N$ is a split extension of $T^l$ by $C_p$:
\[1\to T^l\to G/N \to C_p\to 1\]
 for some $1\le l \le n$.
 \end{lemma}
\begin{proof}
If $G$ is abelian, then $G\cong T^n\times C_p$ and $G$ itself is a split extension.
Assume that $G$ is non-abelian. Then the $C_p$-action on $T^n$ is non-trivial by Lemma \ref{l2-3}, and 
hence one can see that the fixed-point subgroup $(T^n)^{C_p}$ of $T^n$ coincides with the centre $Z(G)$ of $G$. Clearly for $\al \in G$ with $\pi(\al)=a$, it follows that $\al^{p}\in (T^n)^{C_p}=Z(G)$. 
Consider an extension
\[1\to T^n/Z(G)\to G/Z(G) \overset{\ol \pi}{\to} C_p\to 1.\]
Note that $T^n/Z(G) \cong T^l$ for some $1\le l\le n$ since the $C_p$-action on $T^n
$ is non-trivial. 
This extension splits, indeed, a splitting $s: C_p\to G/Z(G)$ is given by $s(a) = \overline\al \in G/Z(G)$, since ${\ol\al}^p =1 \in G/Z(G)$ and $\ol\pi(\ol\al)=a$.
\end{proof}

By Lemma \ref{l2-4}, it suffices to consider a semi-direct product
\[G_\sig := T^n \rtimes_\sig C_p,\] 
where
$\sig: C_p\to \Aut(T^n)$ is a homomorphism which gives the conjugate $C_p$-action on $T^n$. We introduce the following terminology.
An $n$-torus $T^n$ with a homomorphism $\sig: C_p\to \Aut(T^n)$ is called a \emph{$C_p$-torus}, denoted by $T^n_\sig$. 
If a closed subgroup $H$ of $T^n_\sig$ is $C_p$-invariant, then $H$ is called a \emph{$C_p$-subgroup}. Clearly a $C_p$-subgroup $H$ is a closed normal subgroup of $G_\sig$.
In particular, if a subtorus $T$ of $T^n_\sig$ is $C_p$-invariant,  then
$T$ is called a \emph{$C_p$-subtorus} and denoted by $T_\sig$.

As is well-known, $\Aut(T^n)$ is naturally identified with $\GL_n(\bZ)$. 
In fact, an automorphism $\phi\in \Aut(T^n) = \Aut(\bR^n/\bZ^n)$ corresponds to an isomorphism $\wt \phi: \bR^n\to \bR^n$ such that $\wt\phi(\bZ^n)=\bZ^n$, and this corresponds to an isomorphism $\ol \phi:=\wt\phi|_{\bZ^n} :\bZ^n \to \bZ^n$. 
For example, consider $\Ga_p=T^{p-1}\rtimes_\rho C_p$ introduced before. 
Identifying $T^{p-1}$ with $\bR^{p-1}/\bZ^{p-1}$ by the exponential map
\[\exp([x_1, \dots, x_{p-1}]) = (e^{2\pi\sqrt{-1}x_1}, \dots, e^{2\pi\sqrt{-1}x_n})\in T^{p-1},\quad [x_1, \dots, x_{p-1}]\in \bR^{p-1}/\bZ^{p-1},\]
one sees that $\rho(a)\in \Aut(T^{p-1})$ is represented by the $(p-1)\times (p-1)$ matrix
\begin{equation}\label{e2-3}
A=\begin{pmatrix}
0&0&\cdots&0&-1\\
1&0&\cdots&0&-1\\
0&1&\ddots&\vdots&\vdots\\
\vdots&\ddots&\ddots&0&-1\\
0&\cdots&0&1&-1
\end{pmatrix}
\end{equation}
under the standard basis of $\bZ^{p-1}$.
Thus a $C_p$-torus $T^n_\sig$ induces a $\bZ[C_p]$-module $M_\sigma$ whose underling $\bZ$-module is $\bZ^n$  such that the action of $C_p$ on $M_\sig$ is given by $a\cdot x = \ol{\sig(a)}x$ for $x\in M_\sig$.

Conversely, a $\bZ[C_p]$-module $M$ whose underling $\bZ$-module is $\bZ^n$ provides a homomorphism $\psi: C_p\to \GL_n(\bZ)$ induced by the $C_p$-action on $M$. Then $\psi$ induces a homomorphism $\sig: C_p\to \Aut(T^n)$ via the exponential map.  This provides a $C_p$-torus $T_M:=T^n_\sig$ and a semi-direct product $G_M:=G_\sig = T^n\rtimes_\sig C_p$. 
Furthermore,  a $\bZ[C_p]$-submodule $L$ of $M$ provides a $C_p$-subtorus $T_L$ which is a closed normal subgroups of $G_M$. 
Note that, if $\rank_\bZ L= l$, then $\dim T_L =l$ and $G_L$ is a split extension of $T_L$ by $C_p$: \[1\to T_L\to G_L\to C_p\to 1. \]

By integral representation theory, any $\bZ[C_p]$-module $M$ whose underling $\bZ$-module is $\bZ^n$ decomposes into a direct sum of some indecomposable $\bZ[C_p]$-modules (although the decomposition is not unique). 
The indecomposable $\bZ[C_p]$-module is isomorphic to one of following three types, see  \cite[\S 74]{CR} for details.
\begin{enumerate}
\item 
$\bZ$:  the trivial $\bZ[C_p]$-module.
\item
$I$: a fractional ideal $I$ of the cyclotomic field $\bQ(\xi_p)$, where $\xi_p = e^{2\pi\sqrt{-1}/p}$ a primitive $p$-th root of unity.  Note that $I$ is regarded as a $\bZ[C_p]$-module whose underlying $\bZ$-module is $\bZ^{p-1}$ via the natural ring homomorphism $\bZ[C_p]\to \bZ[\xi_p]$.  Furthermore, fractional ideals $I$ and $J$ are isomorphic as $\bZ[C_p]$-modules if and only if these are in the same ideal class of $\bQ(\xi_p)$. For example, the ideal class group of $\bQ(\xi_{23})$ is of order $3$, see \cite{W}, and so there are three isomorphism classes. 
\item
$(I, \nu):=I\oplus \bZ y$ ($\rank_{\bZ}(I,\nu) = p$), where $I$ is a fractional ideal $I$ of $\bQ(\xi_p)$ and $\nu$ is an element of $I$ such that $\nu\not\in (\xi_p-1)I$. The action of $a$ on $y$ is given by 
$ay = \nu + y \in (I,\nu)$. Note that the isomorphism class of $(I, \nu)$ depends only on the ideal class of $I$ and does not depend on the choice of $\nu$. 
\end{enumerate}

\begin{proof}[Proof of Proposition \ref{p2-2}] Suppose that $G=G_\sig$ for some $\sig:C_p\to \Aut(T^{n})$. Let $M_\sig$ be the $\bZ[C_p]$-module corresponding to $T^{n}_\sig\le G_\sig$.
By the indecomposable decomposition of $M_\sig$, one sees that there exists a
$\bZ[C_p]$-submodule $L'$ such that $M_\sig/L' \cong \bZ$, $I$ or $(I, \nu)$.
A $\bZ[C_p]$-module $I$ is a $\bZ[C_p]$-submodule of $(I, \nu)$ and $(I, \nu)/I\cong \bZ$ with the trivial $C_p$-action. 
Consequently, there exists a $\bZ[C_p]$-submodule $L$ such that $M_\sig/L \cong \bZ$ or $I$ and such $L$ provides a $C_p$-subtorus $T_L$ which is a closed normal subgroup of $G_\sig$. If $M_\sig/L \cong \bZ$, then it follows that 
$G_\sig/T_L \cong G_\bZ\cong S^1\times C_p.$
Next suppose that $M_\sig/L \cong I$. It then follows that
$G_\sig/T_L \cong G_I.$
Let $I_0 = \bZ[\xi_p]$ with $\bZ$-basis 
\[\{\xi_p, \xi_p^2, \dots, \xi_p^{p-1}\}.\]
Since the action of the generator $a\in C_p$ on $I_0$ is given by
the multiplication of $\xi_p$, the automorphism induced by the action of $a$ on $I_0$ is represented by the matrix
(\ref{e2-3}).  Hence it follows that $G_{I_0}$ is isomorphic to $\Ga_p$.
Since any ideal class is represented by an integral ideal of $\bZ[\xi_p]$, one may assume that 
$I \subset I_0$ and $I$ is a $\bZ[C_p]$-submodule of $I_0$, and $I$ and $I_0$ have the same $\bZ$-rank $p-1$. Therefore the inclusion $\iota: I \to I_0$ induces an $\bR[C_p]$-isomorphism 
\[\varphi:=\bR\otimes_\bZ \iota: \bR\otimes_{\bZ} I\to \bR\otimes_{\bZ} I_0\] 
such that $\varphi(I)\subset I_0$.
Note that $\bR\otimes_{\bZ} I \cong \bR\otimes_{\bZ} I_0 \cong \bR^{p-1}$ as $\bR$-vector spaces.
Then $\varphi$ induces a surjective $C_p$-homomorphism
\[\ol \varphi: T_I=\bR\otimes_{\bZ} I/I \cong T^{p-1} \to T_{I_0}=\bR\otimes_{\bZ} I_0/I_0\cong T^{p-1}\]
between $C_p$-tori.
This also induces a surjective homomorphism
$f: G_I\to G_{I_0}$ by setting $f(ta^k)= \ol \varphi(t)a^k$ for $t\in T^{p-1}$ and $0\le k\le p-1$. 
Since the kernel $\Ker f = \Ker \ol\varphi$ is a finite $C_p$-subgroup, it follows that $\Ker f$ is a finite normal subgroup of $G_I$. This implies that 
$G_I/\Ker f \cong G_{I_0}\cong \Ga_p$.  
Thus the proof of Proposition \ref{p2-2} is completed.
\end{proof}

\section{The case of $\Gamma_p$}
 So far we have shown that Proposition \ref{p1-1} is reduced to the cases of $\Ga_p$ and $S^1\times C_p$.
In this section, we shall show that  $\Gamma_p$ for any prime $p$ is an anti-BU-group of both types I and \II.
First observe the following:
\begin{proposition}\label{p3-1} 
If a compact Lie group $G$ is an anti-BU-group of type I, then $G$ is an anti-BU-group of type \II.
\end{proposition}
\begin{proof}
Suppose that there exists a $G$-map $f: S(V)\to S(W)$ with 
\[v:=\dim V>w:=\dim W\]
 for some fixed-point-free representations. 
Then one can define a $G$-map
\[h: S(wV) \overset{* f}{\to} S(wW) \overset{\text{incl.}}{\to} S(vW),\]
where $* f$ denote the $w$-fold join of $f:S(V)\to S(W)$.
Since 
\[\dim wV = \dim vW= vw> \dim wW = w^2,\] 
it follows that $\deg h=0$ for dimensional reason. Thus $G$ is an anti-BU-group of type \II.
\end{proof}

In order to construct a $\Ga_p$-map $f: S(V) \to S(W)$ for some representations $V$, $W$ with $\dim V>\dim W$, we use the following fact from equivariant obstruction theory. 

\begin{proposition}\label{p3-2}
Let $G$ be a compact Lie group and $W$ a $G$-representation. Let $X$ be a finite $G$-CW complex and $A$ a $G$-subcomplex of $X$ (possibly empty).
Suppose that there exists a $G$-map $f_A : A\to S(W)$. Let $Y= X\smallsetminus A$. If 
\begin{equation}
\dim Y^H/W_G(H) \le \dim S(W)^H
\end{equation}
for any isotropy subgroup $H$ of $Y$, then there exists a $G$-map $f: X\to S(W)$ extending $f_A$. Here $W_G(H)$ denotes $N_G(H)/H$, and $N_G(H)$ is the normaliser of $H$ in $G$.
\end{proposition}

\begin{proof}
This is a consequence of equivariant obstruction theory  \cite[Chapter \II, 3]{tD}. indeed, since $S(W)^H$ is $(\dim S(W)^H-1)$-connected, there are no obstructions to an extension of the $G$-map $f_A$. 
\end{proof}

We first consider the case of $p=2$. Then $\Ga_2$ is isomorphic to the orthogonal group $\orth(2)$ in dimension $2$. 
The orthogonal group $\orth(2)$ has the orthogonal $2$-dimensional irreducible $\orth(2)$-representations $U'_k$, $k\in \bZ$, whose
underling space is $\bR^2\cong\bC$ and $t\in S^1$ acts by $t\cdot z = t^kz$, $z\in \bC$, and the generator $a\in C_2$ acts by
$a\cdot z = \bar z$ the complex conjugate. 
There are $1$-dimensional $\orth(2)$-representations $\bR$ and $V'_1$, where $\bR$ is the trivial representation and $V'_1$ is given by the lift of the non-trivial $1$-dimensional $C_2$-representation, i.e., $S^1$ acts trivially on $V'_1$ and $a=-1$ acts by $a\cdot x =-x$, $x\in V'_1$.
The following shows that $\Ga_2$ is an anti-BU-group of type I, and hence of type \II.
\begin{lemma}\label{l3-3}
There exists a $\Ga_2$-map $f: S(U'_1\oplus U'_1)\to S(U'_2\oplus V'_1)$.
\end{lemma}
\begin{proof}
Set $V = U'_1\oplus U'_1$ and $W= U'_2\oplus V'_1$.
Then $S(V)$ has two isotropy types $(1)$ and $(C_2)$. 
In fact, the set $\Iso(S(U'_1))$ of isotropy subgroups of $U'_1$ consists of the subgroups $\gr{ta}$, $t\in S^1$. Note that $\gr{ta}$ is conjugate to $C_2$. 
For any $x=(z,w)\in S(U'_1\oplus U'_1)$, the isotropy subgroup $G_x$ of $x$ is equal to $G_z\cap G_w$ for $z$, $w \in U'_1$. 
This deduces that the set of conjugacy classes of isotropy subgroups of $S(V)$ is  
\[\Iso(S(V))/\Ga_2 =\{(1), (C_p)\}.\]
Let $K=W_{\Ga_2}(C_2)$. Since the normaliser $N_{\Ga_2}(C_2)$ is 
$Z\times C_2$, where $Z=\{\pm 1\}$ is the centre of $\Ga_2$, hence $K \cong Z$. Then $S(V)^{C_2}$ is a free $K$-sphere of dimension $1$. Indeed, $V^{C_2}\cong \bR^2$ and $K$ acts antipodally on $S(V)^{C_2}$. On the other hand, $S(W)^{C_2}=S(U_2')^{C_2}\cong S^0$ has the trivial $K$-action. 
Take a constant map 
\[f_{C_2}: S(V)^{C_2} \to S(W)^{C_2}, \]
which is $K$-equivariant. 
Taking the $\Ga_2$-orbits of $S(V)^{C_2}$ and $S(W)^{C_2}$, 
one obtains a (well-defined) $\Ga_2$-map  
\[f_{(C_2)} : \Ga_2S(V)^{C_2} \to \Ga_2S(W)^{C_2}\]
which is defined by 
$f_{(C_2)}(gx) = gf_{C_p}(x)$ for $g\in \Ga_2$, $x\in S(V)^{C_2}.$

One sees that $S(V)\smallsetminus \Ga_2S(V)^{C_2}$ is a free $\Ga_2$-space and
$\dim S(V)/\Ga_2 = 2 = \dim S(W)$. By Proposition \ref{p3-2}, 
there exists a $\Ga_2$-map $f: S(V)\to S(W)$.
\end{proof}
A similar argument is valid for $\Ga_p$, where $p$ is an odd prime. 
We shall summarise facts on $\Ga_p$ here.
\begin{lemma}\label{l3-4}
Let $\Ga_p = T^{p-1}\rtimes_\rho C_p$ as before.
\begin{enumerate}
\item 
For any $t\in T^{p-1}$, the order of $ta$ is $p$ and $\gr{ta}$ is conjugate to $C_p=\gr{a}$.
\item
The centre $Z(\Ga_p)$ of $\Ga_p$ is 
\[Z(\Ga_p)=\gr{(\xi_p, \xi_p^2, \dots, \xi_p^{p-1})} \le T^{p-1}.\]
\item
The normaliser of $C_p$ is $N_{\Ga_p}(C_p) = Z(\Ga_p)\times C_p$. 

\end{enumerate}
\end{lemma}

\begin{proof}
(1) The automorphism $\rho(a)\in \Aut(T^{p-1})$ is represented by the matrix
$A$ in (\ref{e2-1}).
More generally, $\rho(a^i)$, $1\le i\le p-1$, is represented by
\begin{equation}\label{f3-1}
A^i=\begin{pmatrix}
0&\cdots&0&-1&1&&0\\
\vdots&&\vdots&\vdots&&\ddots&\\
0&\cdots&0&-1&0&&1\\
0&\cdots&0&-1&0&\cdots&0\\
1&&0&-1&0&\cdots&0\\
&\ddots&&\vdots&\vdots&&\vdots\\
0&&1&-1&0&\cdots&0
\end{pmatrix}
\end{equation}
where the $(p-i)$-th column is ${}^t(-1,\dots, -1)$, and $A^p=I$.

Let $X\in \GL_{p-1}(\bZ)$ and $\phi\in \Aut(T^{p-1})$  the automorphism induced by $X$.
Set  $t^X=\phi(t)$ for $t\in T^{p-1}$. When $X=(x_{ij})$, one sees that
\[t^X = (t_1^{x_{11}}t_2^{x_{12}}\cdots t_{p-1}^{x_{1\,p-1}}, \dots\dots, t_1^{x_{p-1\,1}}t_2^{x_{p-1\,2}}\cdots t_{p-1}^{x_{p-1\,p-1}}),\]
and $t^{X+Y} = t^Xt^Y$ for matrices $X$, $Y\in \GL_{p-1}(\bZ)$.

For any $t\in T^{p-1}$, since $at = \rho(a)(t)a = t^Aa$, it follows that
\[(ta)^p = t^{I+A+\cdots+A^{p-1}}a^p = t^O = 1.\] 
Therefore the order of $ta$ is $p$. 
Next we show that $\gr{ta}$ is conjugate to $C_p$. Indeed, 
we set 
\[s_k = t_1^{(p-k)/p}\cdots t_k^{(p-k)/p}t_{k+1}^{-k/p}\cdots t_{p-1}^{-k/p}\in S^1\]
for $1\le k\le p-1$, and $s=(s_1,\dots, s_{p-1})\in T^{p-1}$. 
Then $sas^{-1}=s(s^A)^{-1}a$, and 
\[s(s^A)^{-1}=(s_1s_{p-1}, s_2s_1^{-1}s_{p-1}, \dots, s_{p-1}s_{p-2}^{-1}s_{p-1}).\]
By a direct computation, one can see that 
$s(s^A)^{-1}=t$ and $sas^{-1}=ta$. Thus $\gr{ta}$ is conjugate to $C_p$.

(2) Since 
$ata^{-1}=t$ if and only if $t^A=t$, it follows that
\[t_{p-1}^{-1} = t_1,\quad t_1t_{p-1}^{-1}=t_2, \dots\dots, t_{p-2}t_{p-1}^{-1}=t_{p-1}.\]
These equations imply that $t_k = t_1^k$, $	1 \le k\le p-1$, and $t_1^p=1$.
Thus
\[Z(\Ga_p)=\gr{(\xi_p, \xi_p^2, \dots, \xi_p^{p-1})}.\]

(3) Clearly $C_p\le N_{\Ga_p}(C_p)$. If $t\in N_{\Ga_p}(C_p)$ for $t\in T^{p-1}$, then $t^{-1}at = a^k$ for some $1\le k \le p-1$, and this implies that $t^{-1}ata^{-1} = t^{-1}t^A= a^{k-1}\in T^{p-1}\cap C_p = \{1\}$, and so $k=1$ and $t^{-1}t^A =1$. Therefore $t\in Z(\Ga_p)$. Thus the desired result holds.
\end{proof}

The irreducible unitary $\Ga_p$-representations are obtained from the argument of \cite{S}.
Consequently, these are given by the induced representations of non-trivial irreducible $T^{p-1}$-representations and the lifts of irreducible $C_p$-representations.
In the following, we only consider specific representations below.
For any $k\in \bZ\smallsetminus\{0\}$, an irreducible unitary $\Ga_p$-representation $U_k$ is given by
\[U_k = \Ind_{T^{p-1}}^{\Ga_p}\ol U_k, \]
where $\ol U_k$ is a $1$-dimensional unitary $T^{p-1}$-representation on which $t=(t_1,\dots, t_{p-1})\in T^{p-1}$ acts by $t\cdot z = t_1^kz$ for $z\in \ol U_k$. 
Regarding $U_k$ as a direct sum $\oplus_{i=0}^{p-1}a^i\ol U_k$, one sees  that
$a$ acts by permutation of components: $a\cdot a^i\ol U_k = a^{i+1}\ol U_k$, and $t$ acts on $a^i\ol U_k$ by 
\[
t\cdot w_i = a^i(t^{A^{-i}})_1^kz_i \in a^i\ol U_k\]
for $w_i =a^iz_i\in a^i\ol U_k$, $z_i\in \ol U_k$, where $(t^{A^{-i}})_1$ denotes the first component of $t^{A^{-i}}\in T^{p-1}$. More concretely, it follows from the matrix (\ref{f3-1}) that, for any $w=(w_0, w_1,\dots, w_{p-1})\in \oplus_{i=0}^{p-1}a^i\ol U_k$,  
\begin{equation}\label{f3-2}
t\cdot w = (t_1^kz_0,\  at_{1}^{-k}t_{2}^kz_1,\  \dots, a^{p-2}t_{p-2}^{-k}t_{p-1}^kz_{p-2},\  a^{p-1}t_{p-1}^{-k}z_{p-1}).
\end{equation}

For any $k\in \bZ/p$, an irreducible unitary $\Ga_p$-representation $V_k$ is given by the lift of a $1$-dimensional unitary $C_p$-representation $\ol V_k$ on which $a$ acts by $a\cdot z = \xi_p^kz$ for $z\in \ol V_k$. In particular, $T^{p-1}$ acts trivially on $V_k$ and also $V_0$ is the trivial $\Ga_p$-representation.


In general, for an arbitrary unitary $G$-representation $V$, the kernel $\Ker V$ of $V$ is defined by the kernel of the representation homomorphism $\varphi: G\to \text{U}(n)$ of $V$, or equivalently, $\Ker V$ is the closed subgroup consisting of elements $g\in G$ trivially acting on $V$. We note the following:

\begin{lemma}\label{l3-5}
\begin{enumerate}
\item 
For any $k\ge 1$, $\Ker U_k = \Ker U_{-k}=\bZ_k{}^{p-1} \le T^{p-1}$, where $\bZ_k= \gr{\xi_k}\le S^1$ and for any $k\in \bZ/p\smallsetminus \{0\}$, $\Ker V_k = T^{p-1}$.
\item
The centre $Z(\Ga_p)$ is a subgroup of $\Ker U_p$.
\end{enumerate}
\end{lemma}

\begin{proof}
(1) By formula (\ref{f3-2}), the first result is verified. The second result is trivial by definition of $V_k$.

(2) Since $Z(\Ga_p) = \gr{(\xi_p, \dots, \xi_p^{p-1})}$ by Lemma \ref{l3-4}, this is clear by (1).
\end{proof}
\begin{remark}\label{rem-2}
The $C_p$-homomorphism $\varphi: T^{p-1} \to T^{p-1}$, $f(t)=t^k$ induces the isomorphism $\Ga_p/\bZ_k{}^{p-1}\cong \Ga_p$. 
Since $\Ker U_k = \bZ_k{}^{p-1}$, the fixed-point representation $U_k^{\bZ_k{}^{p-1}}$ is regarded as a $\Ga_p$-representation and then $U_k^{\bZ_k{}^{p-1}} \cong U_1$ as $\Ga_p$-representations. Conversely, $U_k$ is regarded as the lift of
$U_1$ by the projection $q: \Ga_p \to \Ga_p/\bZ_k{}^{p-1}\cong \Ga_p$.
\end{remark}
Next we summarise some facts on the isotropy subgroups of $S(U_1)$ and $S(V_1)$.
\begin{lemma}\label{l3-6}
\begin{enumerate}
\item 
Any subgroup $K$ conjugate to $C_p$ is a maximal isotropy subgroup of $S(U_1)$.
\item
Any isotropy subgroup $K$ of $S(U_1)$ not conjugate to $C_p$ is a subgroup of $T^{p-1}$.
\item
For any $x\in S(V_k)$, $k\in \bZ_p\smallsetminus \{0\}$, the isotropy subgroup  $(\Ga_p)_x$ is $T^{p-1}$.
\end{enumerate}
\end{lemma}
\begin{proof}
(1) We may assume that $K=C_p$, since the set $\Iso(S(U_1))$ of isotropy subgroups is closed under conjugation. By definition of the $C_p$-action on $U_1$, we have 
\[U_1^{C_p} = \{(z, az \dots, a^{p-1}z)\in \oplus_{i=0}^{p-1}a^i\ol U_k\,|\,z\in \bC\}\cong \bC.\]
Let $H=(\Ga_p)_u$ be the isotropy subgroup at $u\in S(U_1)^{C_p}$.
Clearly $H\ge C_p$ and $H$ forms an extension 
\[1\to H\cap T^{p-1}\to H\to C_p\to 1.\]
Since $S(U_1)^H\ne\emptyset$ and $S(U_1)^H \subset S(U_1)^{C_p}$, it follows that $S(U_1)^H=S(U_1)^{C_p}\cong S(\bC)$.  
For any $t\in H\cap T^{p-1}$ and $u=(z, az \dots, a^{p-1}z)\in S(U_1)^H$, one sees
\[t\cdot u= (t_1z,\  at_{1}^{-1}t_{2}z,\  \dots, a^{p-2}t_{p-2}^{-1}t_{p-1}z,\  a^{p-1}t_{p-1}^{-1}z)\in S(U_1)^H
 \]
 by formula (\ref{f3-2}).
Since $t\cdot u = u$ and $z\ne 0$, it follows that $t=1$ and hence $H\cap T^{p-1}=1$. Thus $H=C_p$ and $C_p$ is a maximal isotropy subgroup.  

(2) Suppose that $H$ is an isotropy subgroup of $S(U_1)$.
If $\pi(H)= C_p$, where $\pi: \Ga_p\to C_p$ is the projection, then there exists an element $ta\in H$ for some $t\in T^{p-1}$.
Since $C_p'=\gr{ta}\le H$ is is a maximal isotropy subgroup by Lemma \ref{l3-4}. The maximality implies that $H=C_p'$ and $H$ must be conjugate to $C_p$. This contradicts that $H$ is not conjugate to $C_p$ by assumption. It thus follows that $\pi(H)= 1$ and so $H$ is a subgroup of $T^{p-1}$.

(3) Since $\Ker V_k = T^{p-1}$ and $C_p$ acts freely on $S(V_k)$, it follows that
$(\Ga_p)_x = T^{p-1}$.
\end{proof}

\begin{remark}\label{rem-3}
Any isotropy subgroup of $S(U_1)$ included in $T^{p-1}$ is isomorphic to an $m$-torus $T^m$ for some $0\le m <p-1$. 
\end{remark}

The proof of Proposition \ref{p3-1} is finished by the next lemma.
\begin{lemma}
For any odd prime $p$, 
there exists a $\Ga_p$-map 
\[f: S(U_1\oplus U_1)\to S(U_p\oplus (p-1)V_1).\]
\end{lemma}
\begin{proof}
Set $V=U_1\oplus U_1$ and $W=U_p\oplus (p-1)V_1$.
Observing that the isotropy subgroup of $(u_1, u_2)\in V$ is 
\[(\Ga_p)_{(u_1, u_2)} = (\Ga_p)_{u_1}\cap (\Ga_p)_{u_2},\]
one can see that any subgroup $K$ conjugate to $C_p$ is a maximal isotropy subgroup of $S(V)$ by Lemma \ref{l3-6}.
Set $K:= W_{\Ga_p}(C_p)$, which is isomorphic to $Z(\Ga_p)$ by Lemma \ref{l3-4}. By formula (\ref{f3-2}), it follows that $S(V)^{C_p} \cong S^3$ is a free $K$-sphere of dimension $3$, and also $S(W)^{C_p}\cong S^1$ has the trivial $K$-action, since $Z(\Ga_p) \le \Ker W$.
Take a constant map 
$f_{C_p}: S(V)^{C_p} \to S(W)^{C_p}.$
 Then $f_{C_p}$ is $K$-equivariant and so
one can obtain a $\Ga_p$-map 
\[f_{(C_p)}: \Ga_pS(V)^{C_p} \to \Ga_pS(W)^{C_p}.\]
Since any isotropy subgroup $H$ of $Y:=S(V)\smallsetminus \Ga_pS(V)^{C_p}$ is not conjugate to $C_p$, it follows from Lemma \ref{l3-6} that $H$ is a closed subgroup of $T^{p-1}$.
Then for any $H \in \Iso(Y)$, we shall verify the condition of Proposition \ref{p3-2}:
\begin{equation}\label{e3-4}
\dim Y^H/W_{\Ga_p}(H) \le \dim S(W)^H.
\end{equation}
\begin{assertion}\label{a1}
It holds that $\dim_\bR U_1^H \le \dim_\bR U_p^H$ for any non-trivial closed subgroup $H\le T^{p-1}$. 
\end{assertion}
Indeed, by the isomorphism $\varphi:\Ga_p/\bZ_p^{p-1}\to \Ga_p$ defined by $\varphi(\ol t a^i)= t^pa^i$, it follows that
$U_p^{\bZ_p^{p-1}}$ is isomorphic to $U_1$, see Remark \ref{rem-2}. Then one sees
\[\dim_\bR U_p^{H} = \dim_\bR U_p^{H\bZ_p^{p-1}} = \dim_\bR (U_p^{\bZ_p^{p-1}})^{H\bZ_p^{p-1}/\bZ_p^{p-1}} = \dim_\bR U_1^{H^p} \ge \dim_\bR U_1^H, \]
where $H^p = \{t^p\in H\,|\,t\in H\} \le H$. Thus Assertion \ref{a1} holds.
\begin{remark}
If $H$ is a subtorus of $T^{p-1}$, then $H^p=H$ and $\dim_\bR U_1^H=\dim_\bR U_p^H$, see Remark \ref{rem-2}.
\end{remark}
Since $\dim_\bR U_1^H \le 2p-2$ for $H\ne 1$, one sees that
\[\dim_\bR V^H \le \dim_\bR U_1^H + \dim_\bR U_p^H \le 2p-2+\dim_\bR U_p^H =\dim_\bR W^H.\]
This inequality shows that the inequality (\ref{e3-4}) holds for $H\ne 1$.
When $H=1$, it follows that
\[\dim Y/\Ga_p = \dim S(V)/\Ga_p  = 3p \le \dim S(W)= 4p-3,\] 
since $p\ge 3$.
Thus there exists a $\Ga_p$-map $f$ extending $f_{(C_p)}$ by Proposition \ref{p3-2}.
\end{proof}

\begin{remark}
In case of $p=2$, it still follows that there exists a $\Ga_2$-map 
\[f: S(U_1\oplus U_1) \to S(U_2\oplus V_1).\]
Indeed, by Lemma \ref{l3-3}, there exits a $\Ga_2$-map $f: S(U'_1\oplus U'_1) \to S(U'_2\oplus V'_1)$
between orthogonal representation spheres. By complexification, one sees that 
$U_k = \bC\otimes U'_k$ and $V_k = \bC\otimes V'_k$, and $f$ induces a $\Ga_2$-map $f_\bC: S(U_1\oplus U_1) \to S(U_2\oplus V_1)$.
\end{remark}

\section{The case of $S^1\times C_p$}
In this section, we shall show that $G = S^1\times C_p$ is an anti-BU-group of types I and 
$\II$, and complete the proof of Proposition \ref{p2-1}.
In this case, since the obstruction to extension of a $G$-map may appear, the proof is more complicated.
As the first step,  using an argument similar to that in \cite{N1}, we shall show the following result.

\begin{proposition}\label{p5-1}
The group $G= S^1\times C_p$ is an anti-BU-group of type \II. Namely, there exists a $G$-map $f :S(V)\to S(W)$ with $\deg f =0$ for some fixed-point-free $G$-representations $V$ and $W$ with the same dimension.
\end{proposition}

Let $G=S^1\times C_p$ and $a\in C_p$ be a generator of $C_p$ as before. The irreducible unitary $G$-representations are given as follows. 
Let $V_{1,0}$ denote the lift of the $1$-dimensional unitary $S^1$-representation $\ol U_1$ on which $t\in S^1$ acts by $t\cdot z = tz$ for 
$z\in V_{1,0}$. Let $V_{0,1}$ denote the lift of the $1$-dimensional unitary $C_p$-representation $\ol V_1$ on which $a$ acts by $a\cdot z = \xi_pz$ for $z\in V_{0,1}$. Every irreducible unitary $G$-representation is given by $V_{k,l}:=V_{1,0}^{\otimes k}\otimes V_{0,1}^{\otimes l}$, $k\in \bZ$, $l\in \bZ/p$.
We consider $G$-representations $V_{1,0}\oplus V_{1,1}$ and $V_{p,0}\oplus V_{0,1}$.
\begin{lemma}\label{l4-2}
For $V=V_{1,0}\oplus V_{1,1}$ and $W=V_{p,0}\oplus V_{0,1}$, there exists a $G$-map 
$h: S(V) \to S(W)$.
\end{lemma}

\begin{proof}
Note that $\Ker V_{1,0}=C_p=\gr{a}$ and $\Ker V_{1,1} = C_p':= \gr{\xi_p^{-1}a}$.
Therefore $\Iso(S(V))$ consists of $C_p$, $C_p'$ and $1$. 
Similarly, $\Ker V_{p,0} = \bZ_p\times C_p$, where $\bZ_p = \gr{\xi_p} \le S^1$ and $\Ker V_{0,1} =S^1$, and $\Iso(S(W))$ consists of $\bZ_p\times C_p$, $S^1$ and $\bZ_p$.
A $G$-map
\[h_{C_p}: S(V)^{C_p}= S(V_{1,0}) \to  S(W)\]
 is defined by $h_{C_p}(z)=(z^p, 0)$ for $z\in S(V_{1,0})$, 
and also a $G$-map
\[h_{C'_p}: S(V)^{C'_p}= S(V_{1,1}) \to  S(W)\]
is defined by $h_{C'_p}(w)=(w^p, 0)$ for $w\in S(V_{1,1})$.
Using these maps, we obtain a $G$-map 
\[h^{>1}: S(V)^{>1}\to S(W),\]
 where $S(V)^{>1}$ is 
the singular set: 
\[S(V)^{>1}:= \{x\in S(V)\,|\,G_x \ne 1\} = S(V)^{C_p}\coprod S(V)^{C'_p}. \]
Since $G$ acts freely on $S(V)\smallsetminus S(V)^{>1}$ and $\dim (S(V)\smallsetminus S(V)^{>1})/G = 2< \dim S(W) =3$, there exists a $G$-map $h: S(V)\to S(W)$ extending $h^{>1}$ by Proposition \ref{p3-2}.
\end{proof}

Next we shall show $\deg h =0$ for any $G$-map $h$ as above and finish the proof of Proposition \ref{p5-1}. To do that, we use the Euler classes of an oriented orthogonal representation $V$, i.e, the $G$-action on $V$ is orientation preserving under a given orientation on $V$. 
A unitary $G$-representation $V$ has a canonical orientation given by the complex structure of $V$ and oriented as an orthogonal $G$-representation.
Generally, the Euler class 
\[e_G(V)\in H^{n}(BG, R)\]
 of an oriented orthogonal representation $V$ of dimension $n$ is defined to be the Euler class of the associated vector bundle
$\pi: EG\times_GV \to BG$ over the classifying space $BG$. 
In case of a unitary $G$-representation $V$, the Euler class $e_G(V)$ coincides with the top Chern class of the associated complex vector bundle. 
Although the coefficient ring $R$ can be taken to be $\bF_p$, $\bZ$ or $\bQ$, etc, 
we here take $R=\bQ$ as coefficients, because the cohomology ring of $BG$ becomes simpler and it is sufficient for our purpose.
A key result is the following special case of the result of \cite{M2}, see also \cite{N1}.
\begin{proposition}[(\cite{M2}, \cite{N1})]\label{p5-3}
Let $G$ be a compact Lie group and $V$, $W$ fixed-point-free, oriented orthogonal $G$-representations with the same dimension $n$.
Suppose that there exists a $G$-map $h: S(V)\to S(W)$. Then
\[e_G(W) = (\deg h)e_G(V) \in H^{n}(BG; R).\] 
In particular, under $R=\bQ$, if $e_G(V)\ne 0$ and $e_G(W)= 0$, then $\deg h =0$. 
\end{proposition}
Now we return to the case of $G=S^1\times C_p$. The next lemma shows Proposition \ref{p5-1}.
\begin{lemma}\label{l4-4}
Let $V=V_{10}\oplus V_{11}$ and $W= V_{p,0}\oplus V_{0,1}$.
For any $G$-map $h: S(V) \to S(W)$, it follows that $\deg h =0$.
\end{lemma}
\begin{proof}
Since $BG\simeq BS^1\times BC_p$ and $BC_p$ is $\bQ$-acyclic, 
it follows that 
\[\text{Res}_{S^1}=i^*: H^*(BG;\bQ) \to H^{*}(BS^1; \bQ)\]
is a graded ring isomorphism, where $i: S^1\to G$ is the natural inclusion.
Since 
$H^{*}(BS^1; \bQ)\cong \bQ[c]$ as graded rings,
where $c$ is the first Chern class of the canonical complex line bundle over $BS^1\simeq \bC P^\infty$, one obtains $H^*(BG;\bQ)\cong \bQ[c]$ as graded rings.   
Clearly $\text{Res}_{S^1}V_{1,0}$ and $\text{Res}_{S^1}V_{1,1}$ are isomorphic to the standard $S^1$-representation $\ol U_1$ whose associate complex vector bundle is isomorphic to the canonical one over $BS^1$, hence $e_G(V) = e_G(V_{10})e_G(V_{11})=c^2 \ne 0$. On the other hand, $\text{Res}_{S^1}V_{0,1}$ is the trivial $S^1$-representation, hence $e_G(V_{0,1})=0$. This implies that $e_G(W) = e_G(V_{p,0})e_G(V_{0,1}) = 0$.
Thus $\deg h =0$ by Proposition \ref{p5-3}.
\end{proof}
The second step is to construct a $G$-map $f:S(V) \to S(W)$ for some fixed-point-free representations $V$, $W$ with $\dim V>\dim W$. We shall prove this using a $G$-map $h$ of degree $0$.  
The goal is to prove the following result.
\begin{proposition}\label{p4-5}
The group $G=S^1\times C_p$ is an anti-BU-group of type I. Namely, there exists a $G$-map $f:S(V) \to S(W)$ for some fixed-point-free representations $V$, $W$ with  $\dim V>\dim W$.
\end{proposition}
In order to prove this, we again use equivariant obstruction theory \cite[Chapter \II, 3]{tD}. In this case the obstruction may appear; however, the computation of the obstruction class is not easy in general. In order to avoid this difficulty, we use the existence of a $G$-map of degree $0$.
This idea is based on an argument of \cite{Wa}.

First recall the equivariant primary obstruction class. Let $G$ be a compact Lie group.  
Let $X$ a finite $G$-CW complex and $Y$ an $n$-simple and $(n-1)$-connected $G$-space, $n\ge 1$. 
Let $X^{>1}$ be the singular set of $X$ and suppose that there exists a $G$-map $g: X^{>1}\to Y$. Let $X_{(m)}$ be an $m$-skeleton relative to $X^{>1}$, i.e., $X_{(m)}$
is the union of free $i$-cells $G\times e^i$ of $X\smallsetminus X^{>1}$ for $i\le m$ with $X^{>1}$. Since $Y$ is $(n-1)$-connected, there exists a $G$-map $f_{n}: X_{(n)} \to Y$ extending $g$, since the obstructions to extension vanish. In this situation, the equivariant primary obstruction $\ga(g)$ is defined in the equivariant cohomology group $\fH_G^{n+1}(X, X^{>1}; \pi_{n}(Y))$ and there exists a $G$-map $f_n: X_{(n+1)} \to Y$ extending $g$ if and only if $\ga(g)=0$. 

\begin{lemma}\label{l4-6}
Let $Z$ be another $n$-simple and $(n-1)$-connected $G$-space.
\begin{enumerate}
\item 
For any $G$-map $h: Y\to Z$, it follows that $\ga(h\circ g) = h_\#(\ga(g))$, where 
\[h_\#: \fH_G^{n+1}(X, X^{>1}; \pi_{n}(Y))\to \fH_G^{n+1}(X, X^{>1}; \pi_{n}(Z))\]
is the homomorphism induced by $h$. 
\item
If there exists a $G$-map $h: Y\to Z$ such that $h_*=0 : \pi_n(Y)\to \pi_n(Z)$, then there exists a $G$-map $f_{n+1} : X_{(n+1)}\to Z$ extending $h\circ g$. 
\end{enumerate}
\end{lemma}

\begin{proof}
(1) From \cite[Chapter \II, 3]{tD}, the equivariant primary obstruction $\ga(f)$ is represented by a cocycle
\[c^{n+1}(g): C_{n+1}(X_{(n+1)}, X_{(n)}) \overset{\rho}{\leftarrow} \pi_{n+1}(X_{(n+1)}, X_{(n)})\overset{\partial}{\to} \pi_{n}(X_{(n)}) \overset{f_{n\,*}}{\to} \pi_{n}(Y),\]
where $\rho$ is the Hurewicz homomorphism
\[\rho: \pi_{n+1}(X_{(n+1)}, X_{(n)}) \to H_{n+1}(X_{(n+1)}, X_{(n)};\bZ)=C_{n+1}(X_{(n+1)}, X_{(n)}).\]
Since an extension $f_{n}':  X_{(n)} \to Z$ of $h\circ g$ is given by $f_{n}'= h\circ f_{n}$, the obstruction class $\ga(h\circ g)$ is represented by
\[c^{n+1}(h\circ g): C^{n+1}(X_{(n+1)}, X_{(n)}) \overset{\rho}{\leftarrow} \pi_{n+1}(X_{(n+1)}, X_{(n)})\overset{\partial}{\to} \pi_{n}(X_{(n)}) \overset{(h\circ f_{n})_*}{\longrightarrow} \pi_{n}(Z).\]
Clearly $c^{n+1}(h\circ g) = h_*(c^{n+1}(g))$ and thus 
$\ga(h\circ g)=h_\#(\ga(g))$. 

(2) Since $h_*=0$, the obstruction class $\ga(h\circ g) =h_\#(\ga(g))$ vanishes.
Therefore there exists a $G$-map $f_{n+1} : X_{(n+1)}\to Z$ extending $h\circ g$. 
\end{proof}
\begin{proof}[Proof of Proposition \ref{p4-5}]
Consider $G$-representations
\[V= 2V_{1,0}\oplus 2V_{1,1}, \quad U= V_{1,0}\oplus V_{p,0}\oplus V_{1,1},  \quad W=2V_{p,0}\oplus V_{0,1}.\]
Observing that 
\[V^{C_p} = 2V_{1,0},\quad U^{C_p} = V_{1,0}\oplus V_{p,0},\quad V^{C'_p} = 2V_{1,1},\quad U^{C'_p} = V_{p,0}\oplus V_{1,1},\]
one can define $G$-maps 
\[f^{C_p}: S(V)^{C_p} \to S(U)^{C_p}, (z,w)\mapsto (z, w^p)/\|(z,w^p)\|\]
 and 
 \[f^{C'_p}: S(V)^{C'_p} \to S(U)^{C'_p}, (z,w)\mapsto (z^p, w)/\|(z^p,w)\|.\]
Therefore, there exists a $G$-map $g: S(V)^{>1} \to S(U)$, where 
$S(V)^{>1} = S(V)^{C_p}\coprod S(V)^{C_p'}.$
By Lemmas \ref{l4-2} and \ref{l4-4}, there exists a $G$-map 
\[h: S(V_{1,0}\oplus V_{1,1}) \to S(V_{p,0}\oplus V_{0,1})\]
of degree $0$.
We define a $G$-map $\wt h :S(U)\to S(W)$ by
\[\wt h=h*id: S(U) \cong S(V_{1,0}\oplus V_{1,1})*S(V_{p,0}) \longrightarrow S(V_{p,0}\oplus V_{0,1})*S(V_{p,0}) \cong S(W) \]
where $*$ means join. Then $\wt h_*=0$ on $\pi_5(S(U))$, since $\deg(h*id) =0$.
Since $S(U)$ and $S(W)$ are $4$-connected, it follows from Lemma \ref{l4-6} that
there exists a $G$-map $f: S(V)_{(6)} \to S(W)$ extending $\wt h\circ g$. Since $\dim S(V)/G = 6$,  it follows that $S(V)_{(6)}$ coincides with the whole space $S(V)$. Thus there exists a $G$-map $f:S(V)\to S(W)$. 
\end{proof}
\begin{remark}
In case of $p=2$, by an argument similar to that in Lemma \ref{l3-3}, one can see that there exists an $S^1\times C_2$-map between orthogonal representation spheres as a counterexample of the Borsuk-Ulam property of type I.
\end{remark}

\begin{proof}[Proof of Theorem \ref{t1-2}]
It is already known that statement (3) implies (2), see for example \cite{M0}. Statement (2) implies (1) by Proposition \ref{p3-1}.  The discussion so far shows that if $G$ is neither $C_p{}^p$ nor $T^n$, then $G$ is an anti-BU-group of both types I and \II. In particular, statement (1) implies (3). 
\end{proof}

\noindent Ikumitsu Nagasaki\\
Department of Mathematics\\
Kyoto Prefectural University of Medicine\\
1-5 Shimogamo Hangi-cho\\  
Sakyo-ku 606-0823, Kyoto\\
Japan\\  
email : nagasaki@koto.kpu-m.ac.jp

\end{document}